%% file: Mousley_Teichmuller_geodesics_HHS_final.tex
\documentclass[11pt]{article}

\usepackage{graphicx}
\usepackage{color}
\usepackage{transparent}

\usepackage{geometry}
\usepackage{amsmath}
\usepackage{amsfonts}
\usepackage{amsthm}
\usepackage{amssymb}
\usepackage{mathdots}
\usepackage{latexsym}
\usepackage{mathrsfs}
\usepackage{subfig}
\usepackage{color}
\usepackage{enumerate}
\usepackage{rotating}
\usepackage{multirow}
\usepackage{hyperref}
\usepackage{pgfplots}
\usepackage{lineno, blindtext}
\usepackage{cancel}
\usepackage{wrapfig}
\usepackage{float}
\usepackage{multicol} 
\usepackage{graphicx}
\usepackage{dsfont}
\usepackage{yfonts}
\usetikzlibrary{calc}

\usepackage{cite}

\usepackage{tkz-euclide}
\usetikzlibrary{decorations.markings,arrows}
\tikzset{
    arrowMe/.style={
        postaction=decorate,
        decoration={
            markings,
            mark=at position .5 with {\arrow{#1}}
        }
    }
}

\newtheorem{theorem}{Theorem}[section]
\newtheorem{lemma}[theorem]{Lemma}

\newtheorem{remark}[theorem]{Remark}

\newtheorem{convention}[theorem]{Convention}

\newcommand{\C}{\mathcal{C}}

\newcommand{\Mod}{{\rm Mod}}

\newcommand{\ol}{\overline}
\newcommand{\T}{\mathcal{T}}
\newcommand{\G}{\mathcal{G}}
\newcommand{\R}{\mathbb{R}}
\newcommand{\Ho}{\mathcal{H}}
\newcommand{\base}{\text{base}}
\newcommand{\simp}{\bigtriangleup^2}
\newcommand{\cL}{\mathcal{L}}
 
\newcommand{\Ext}{{\rm Ext}}
\newcommand{\Hyp}{{\rm Hyp}}

\def\co{\colon\thinspace}

\newcommand{\hgline}[2]{
\pgfmathsetmacro{\thetaone}{#1}
\pgfmathsetmacro{\thetatwo}{#2}
\pgfmathsetmacro{\theta}{(\thetaone+\thetatwo)/2}
\pgfmathsetmacro{\phi}{abs(\thetaone-\thetatwo)/2}
\pgfmathsetmacro{\close}{less(abs(\phi-90),0.0001)}
\ifdim \close pt = 1pt
    \draw[black] (\theta+180:1) -- (\theta:1);
\else
    \pgfmathsetmacro{\R}{tan(\phi)}
    \pgfmathsetmacro{\distance}{sqrt(1+\R^2)}
    \draw[black] (\theta:\distance) circle (\R);
\fi
}

\newcommand{\hglinewhite}[2]{
\pgfmathsetmacro{\thetaone}{#1}
\pgfmathsetmacro{\thetatwo}{#2}
\pgfmathsetmacro{\theta}{(\thetaone+\thetatwo)/2}
\pgfmathsetmacro{\phi}{abs(\thetaone-\thetatwo)/2}
\pgfmathsetmacro{\close}{less(abs(\phi-90),0.0001)}
\ifdim \close pt = 1pt
    \draw[black] (\theta+180:1) -- (\theta:1);
\else
    \pgfmathsetmacro{\R}{tan(\phi)}
    \pgfmathsetmacro{\distance}{sqrt(1+\R^2)}
    \fill[white] (\theta:\distance) circle (\R);
\fi
}

\begin{document}

\title{Exotic limit sets of  Teichm\"uller \\ geodesics in the HHS boundary}
\author{Sarah C. Mousley}
\date{}

\maketitle

\abstract{We answer a question of Durham, Hagen, and Sisto, proving that a Teichm\"uller geodesic ray does not necessarily converge to a unique point in the hierarchically hyperbolic space boundary of Teichm\"uller space. In fact, we prove that the limit set can be almost anything allowed by the topology.

 }

\section{Introduction} \label{sec:intro}

Let $S=S_g$ be a connected, closed, orientable surface of genus  $g \geq 2$, and let
$\T(S)$ denote the Teichm\"uller space of $S$ equipped with the Teichm\"uller metric. Masur \cite{Masur2} proved that $\T(S)$ is not non-positively curved in the sense of Busemann, and  Masur and Wolf \cite{MasurWolf} showed that $\T(S)$ is not (Gromov) hyperbolic. 
In this paper, we explore to what extent $\T(S)$ has features of negative curvature by studying the asymptotic behavior of geodesics.

The notion of Gromov boundary can be  generalized  in several ways to obtain a boundary for $\T(S)$. For example the visual boundary and the Morse boundary agree with the Gromov boundary when the space is hyperbolic, and these boundaries are well-defined for $\T(S)$ (see \cite{McPap} and \cite{Cordes}). Here we consider another generalization. 
Both hyperbolic spaces and $\T(S)$ can be equipped with a geometric structure defined by Behrstock, Hagen, and Sisto \cite{BHS} called a hierarchically hyperbolic space (HHS) structure. (That $\T(S)$ can be equipped with such a structure follows from the results in \cite{Durham}, \cite{EMR},\cite{MMII}, \cite{Rafi_comb}.)
These structures were used by Durham, Hagen, and Sisto \cite{DHS} to construct a boundary, which we will call the HHS boundary (see Section \ref{sec:background} for definitions). 

Working in the HHS paradigm, 
the question  becomes how do the asymptotics of  geodesic rays in the HHS boundary of $\T(S)$ compare to those of geodesic rays in the HHS boundary of a hyperbolic space? 
The identity map on a hyperbolic space extends to a homeomorphism between its HHS and Gromov boundaries, so certainly in this case
geodesic rays are well-behaved.
In \cite{DHS} Durham, Hagen, and Sisto asked for a description of limit sets of Teichm\"uller geodesic rays in the HHS boundary.  Our main result provides an answer to this question.

\begin{theorem} \label{theorem:main}
Given a continuous map $\gamma \co \mathbb{R} \rightarrow \bigtriangleup^2$ to the standard 2-simplex,  there exists  a Teichm\"uller geodesic  ray $\G$ in $\T(S_3)$ and an embedding of $\simp$ into the HHS boundary of $\T(S_3)$  such that the limit set of $\mathcal{G}$ in the HHS boundary  is the image of $ \ol{\gamma(\mathbb{R})}.$
\end{theorem}

The study of limiting behaviors of Teichm\"uller geodesic rays began with Kerckhoff \cite{Kerckhoff}. 
He proved that the  Teichm\"uller boundary of $\T(S)$ (the collection of all geodesic rays emanating from a fixed basepoint) is basepoint dependent. Since then, 
the  limit sets of geodesic rays in Thurston's compactification of $\T(S)$ by $\mathcal{PMF}(S)$, the space of projectivized measured foliations, have received much attention. Masur \cite{Masur3} showed that almost all Teichm\"uller geodesic rays converge to a unique point in $\mathcal{PMF}$. Lenzhen \cite{Lenzhen} provided the first example of a  geodesic ray whose limit set in $\mathcal{PMF}$  is more than one point. The study of limit sets in $\mathcal{PMF}$ continued in  \cite{CMW} and  \cite{LLR}, where the influence of the topological and dynamical properties of the associated vertical foliation is studied, and in \cite{BLMR} and \cite{LMR}, where rays with limits sets homeomorphic to a  circle and 2-simplex are constructed, respectively.
It would be interesting to know whether the kind of behavior we produce in Theorem \ref{theorem:main} can occur in $\mathcal{PMF}$. 
\

\textbf{Strategy for proving Theorem \ref{theorem:main}:} To build Teichm\"uller geodesic rays, we use a construction first described by Masur and Tabachnikov \cite{MT} and  used by  Lenzhen \cite{Lenzhen} and Lenzhen, Modami, Rafi \cite{LMR} to study limit sets of Teichm\"uller geodesics in Thurston's compactification. 
Given irrational numbers $\theta_0, \theta_1, \theta_2$ and $0<s<1$, for each $i=0,1,2$
cut a slit of length $s$ and slope $\theta_i$  in  a unit square $R_i$. For each $R_i$, identify its parallel sides to form a  torus with one boundary component. Then identify the left side of the slit in $R_i$ with the right side of the slit in $R_{i-1}$ (indices mod 3). This produces a genus $3$ translation surface, yielding a complex structure $X$  and a  quadratic differential $q$ with respect to $X$ (see Figure \ref{glued_up}). 
 We consider the Teichm\"uller geodesic ray  corresponding to $(X,q)$.
Given a continuous map $\gamma: \R \rightarrow \simp$, we will show how to construct irrational numbers $\theta_0, \theta_1, \theta_2$ and an embedding of  $\simp$ into the HHS boundary of $\T(S_3)$ so that the limit set  of the corresponding Teichm\"uller geodesic ray is the image of $\ol{\gamma(\R)}$.

\begin{figure} 
 \center
\begin{tikzpicture}[scale=1.5]
\draw (0,0)--(1,0)--(1,1)--(0,1)--(0,0); 
\draw (1.5,0)--(2.5,0)--(2.5,1)--(1.5,1)--(1.5,0); 
\draw (.75,-1.5)--(1.75,-1.5)--(1.75,-.5)--(.75,-.5)--(.75,-1.5);


\node (R0) at (.5, -.15) {$R_0$};
\node (R1) at (2, -.15) {$R_1$}; 
\node (R2) at (1.25, -1.65) {$R_2$}; 

\coordinate (s1bottom) at (0.29721, 0.210386){};
\coordinate (s1top) at (0.70279, 0.789614){};

\coordinate (s2bottom) at (1.75, .25) {};
\coordinate (s2top) at (2.25,.75) {};  

\coordinate (s3bottom) at (0.901818, -1.06139){};
\coordinate (s3top) at  (1.59818,-0.938606){};

\begin{scope}[decoration={
    markings,
    mark=at position 0.5 with {\arrow{>}}}, 
    ] 
\draw [postaction={decorate}] (s1bottom) to [out=80, in=210] (s1top); 
\draw [postaction={decorate}] (s3bottom) to [out=-15, in=215] (s3top); 
\end{scope}

\draw[->] (2.9, -.25) to [out=20, in=160](3.9, -.25);
\node at (3.4, 0) {\small glue}; 

\begin{scope}[decoration={
    markings,
    mark=at position 0.5 with {\arrow{>>}}}, 
    ] 
\draw [postaction={decorate}] (s1bottom) to [out=30, in=255] (s1top); 
\draw [postaction={decorate}] (s2bottom) to [out=70, in=200] (s2top); 
\end{scope}


\begin{scope}[decoration={
    markings,
    mark=at position 0.55 with {\arrow{>>}}}, 
    ] 
\draw [postaction={decorate}] (s2bottom) to [out=20, in=245] (s2top); 
\draw [postaction={decorate}] (s3bottom) to [out=35, in=165] (s3top); 
\end{scope}

\begin{scope}[decoration={
    markings,
    mark=at position 0.4 with {\arrow{>}}}, 
    ] 
\draw [postaction={decorate}] (s2bottom) to [out=20, in=245] (s2top); 
\draw [postaction={decorate}] (s3bottom) to [out=35, in=165] (s3top); 
\end{scope}

\

\end{tikzpicture}
\hspace{-30pt}
 \def\svgwidth{150pt}
 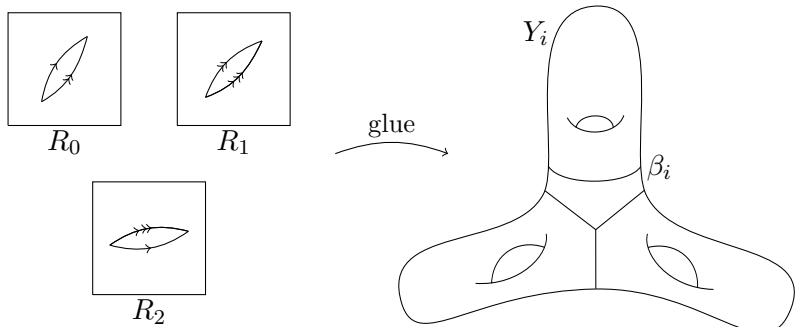
\vspace{-10pt}
\captionsetup{width=.9\linewidth}
 \caption{Three slitted unit squares glued to form a genus 3 translation surface.}
 \label{glued_up} 
 \end{figure}
 
 \

In Section \ref{sec:background} we define necessary terms and introduce notation. In Section \ref{sec:form_acc_pt} we will give conditions on  the irrational numbers to guarantee that the limit set in the HHS boundary of the corresponding Teichm\"uller geodesic ray is contained in a 2-simplex. 
Section \ref{sec:main_result} contains the proof of Theorem \ref{theorem:main}, rephrased there as Theorem \ref{theorem:main_rephrased}.  Section \ref{sec:main_result} shows how to carefully choose the entries of the continued fraction expansions of our irrational numbers to 
obtain fine control of the limit set. 

\

\noindent \textbf{Acknowledgments:} The author was supported by a Research Assistantship through NSF Grant number DMS-1510034.
The author is very grateful to her PhD advisor Chris Leininger for his guidance, support, and candid feedback.  
The author would also like to thank Matthew Durham and Kasra Rafi for helpful conversations.

\section{Background} \label{sec:background} 
In this section, we recall some needed definitions and theorems.  Throughout this paper, we let $S$ denote a connected, closed, orientable surface of genus at least $2$. 

\

\noindent \textbf{Notation: } Let $f, g \co Y \rightarrow \mathbb{R}$ be functions. If there exist constants $A \geq 1$ and $B \geq 0$ that depend only on the topology of $S$, such that for all $y \in Y$, we have 
 $\frac{1}{A}(g(y)-B) \leq f(y) \leq Ag(y)+B$, then we write $f \asymp g$. In the case that $B=0$ we write $f \stackrel{*}{\asymp} g$, and if $A=1$ we write $f \stackrel{+}{\asymp} g$. We define  $\prec, \stackrel{*}{\prec}$, and $\stackrel{+}{\prec}$ similarly. 

\

\subsection{Curve graph, combinatorial horoballs}
A \textit{curve} in $S$ is a homotopy class of an essential, simple, closed path in $S$. We say that two curves \textit{do not intersect} if there is a representative of each such that the representatives are disjoint.  Given a non-annular subsurface $Y$, we define the \textit{curve graph of $Y$}, denoted $\mathcal{C}(Y)$, to be the graph whose vertices are all the curves that have representatives essential in $Y$. Two curves are adjacent if and only if they do not intersect. In the case that $Y$ is a once-punctured torus or four-times punctured sphere, we modify the definition: two curves are adjacent if and only if they have representatives that have minimum intersection (i.e. intersect once or twice, respectively).

To every curve $\alpha$, we associate an annular subsurface $Y_\alpha$ by taking the visual compactification of an annular cover of $S$ to which $\alpha$ lifts. The \textit{curve graph of $Y_\alpha$}, which we denote both by $\C(Y_\alpha)$ and $\C(\alpha)$, is the graph whose vertices are  homotopy classes of embedded arcs in $Y_\alpha$ that connect one boundary component to the other, where the endpoints of the arc must be fixed throughout the homotopy. Two vertices are adjacent if and only if they have representative arcs whose restrictions to $\text{int}(Y_\alpha)$ are disjoint.

We write $Y \subseteq S$ to indicate that $Y$ is a subsurface of $S$, even though when $Y$ is an annulus, $Y$ is not a subset of $S$. 

Given a curve $\alpha$ in $S$, the \emph{combinatorial horoball associated to $\alpha$}, denoted $\Ho_\alpha$,  is the following graph. Begin with the graph Cartesian product $\C(\alpha) \times \mathbb{Z}_{\geq0}$ and then for each $n$ add edges so that each vertex $(x,n)$ is adjacent to every vertex in $\{(y,n):d_{\C(\alpha)}(x,y) \leq e^n\}.$

The spaces $\Ho_\alpha$ and $\C(Y)$ for each subsurface $Y$ are Gromov hyperbolic (see\cite{GM} and  \cite{MasurMinsky1}, respectively). We let $\partial \Ho_\alpha$ and $\partial \C(Y)$ denote their Gromov boundaries. 

\subsection{Extremal length and Teichm\"uller geodesics} 
Here we give some definitions and results from Teichm\"uller theory. We refer the reader to 
 \cite{Pap}, specifically chapter 2, for a careful treatment of the material. 
 
We define the \textit{Teichm\"uller space of $S$}, denoted $\T(S)$, to be the collection of equivalence classes of complex structures on $S$, where we define two complex structures to be equivalent if there is a map $S \rightarrow S$ isotopic to the identity which is biholomorphic when the domain is equipped with one of the complex structures and the range is equipped with the other. 

Consider $X \in \T(S)$. Throughout this paper, when it is convenient, we will also use $X$ to denote a structure in the equivalence class. 
Given an annulus $A$ in $S$, we let $\Mod_X(A)$ denote the modulus of $A$ in $X$. 
We define the \textit{extremal length in $X$ of a curve $\alpha$} in $S$ to be
\[\Ext_X(\alpha)= \inf \frac{1}{\Mod_X(A)},\]
where the infimum is taken over all annuli $A$ in $S$ with core curve $\alpha$. 

Every complex structure determines a collection of conformally equivalent Riemannian metrics on $S$, and in the collection there is a unique hyperbolic metric by the Uniformization Theorem. 
We let $\Hyp_X(\alpha)$ denote the length of the geodesic representative of $\alpha$ in the hyperbolic metric associated to $X \in \T(S)$. 
The following theorem gives a comparison of hyperbolic and extremal lengths, showing that when the hyperbolic length of a curve is small, its extremal length and hyperbolic length are (coarsely) equal. 
\begin{theorem}[Maskit \cite{Maskit}] \label{theorem:maskit} Given $X \in \T(S)$ and a curve $\alpha$ in $S$,
\[\frac{1}{\pi} \leq \frac{\Ext_X(\alpha)}{\Hyp_X(\alpha)} \leq \frac{1}{2} e^{\Hyp_X(\alpha)/2}. \]
\end{theorem}

Let $q$ be a (holomorphic) quadratic differential with respect to $X$, and consider the induced singular flat structure on $S$.  We let $\ell_q(\gamma)$ denote the $q$-length of a geodesic representative of a curve $\gamma$ in the metric induced by $q$.
The collection of $q$-geodesic representatives of a curve $\alpha$  form a  (possibly degenerate) Euclidean cylinder, which we will call $F$. An \emph{expanding annulus with core $\alpha$} is the largest one-sided regular neighborhood of a boundary component of $F$ in a direction away from $F$ that is an embedded annulus. Let $E$ and $G$ denote the two expanding annuli with core $\alpha$. The following theorem of Minsky relates  $\Ext_X(\alpha)$ with the moduli of $F, E, $ and $G$, and the subsequent theorem gives a way to estimate the modulus of an annulus  that satisfies certain properties. 
\begin{theorem}[Minsky \cite{Minsky3}, Theorems 4.5 and 4.6] \label{theorem:Ext_sum_mod}  There exists $\epsilon_0$ depending only on $S$ such that if  $\Ext_X(\alpha) \leq \epsilon_0$, then 
\[\frac{1}{\Ext_X(\alpha)} \stackrel{*}{\asymp} \Mod_X(E)+\Mod_X(F)+\Mod_X(G).\]
\end{theorem}

\begin{theorem}[Rafi \cite{Rafi_short} Lemma 3.6] \label{theorem:modulus} Let $q$ be a quadratic differential with respect to $X \in \T(S)$. Let $A$ be an annulus in $S$ such that with respect to the $q$-metric, $A$ has equidistant boundary components and exactly one boundary component $\gamma_0$ a geodesic. Further suppose the interior of $A$ does not contain any singularities of $q$. 
Then 
\[\emph{Mod}_X(A) \asymp \log \left(\frac{d}{\ell_q(\gamma_0)}\right),\]
where $d$ is the $q$-distance between the boundary components of $A$.
\end{theorem}

Let $q$ a quadratic differential with respect to $X$. The pair $(X,q)$ determines a  geodesic in the Teichm\"uller metric.   Composing the natural coordinates of  $q$  away from its singularities   with $\left( {\begin{array}{cc}
   e^t & 0 \\       0 & e^{-t} \\      \end{array} } \right)$ yields a new complex structure $X_t \in \T(S)$ on $S$ and a new quadratic differential $q_t$ with respect to $X_t$. The map $\G:(-\infty, \infty) \rightarrow \T(S)$ given by $t \mapsto X_t$ is a  geodesic.  All geodesics in $\T(S)$ can be described in this way. 
  
 We will let $\Ext_t$, $\Mod_t$, and $\Hyp_t$ denote  $\Ext_{X_t}$,$\Mod_{X_t}$, and $\Hyp_t$,  respectively.
Let $\alpha$ be a curve in $S$ such that no representative of $\alpha$ is a leaf of the vertical or horizontal foliation of $S$ corresponding to $q$. Then we define the \emph{balance time of $\alpha$} to be the time $t$ such that the horizontal $q_t$-length and vertical $q_t$-length of $\alpha$ are equal. 
We define the  \textit{geodesic ray determined by $(X,q)$} to be $\G$ restricted to $[0,\infty)$. 

\subsection{HHS structure of Teichm\"uller space}

Given  $X \in \T(S)$, we define a \emph{short marking} $\mu_X$ on $X$ to be a maximal collection of disjoint curves with associated transversals selected as follows. Choose a curve $\alpha_1$ with shortest extremal length in $X$, then of those curves that do not intersect $\alpha_1$, choose one with shortest extremal length. Continue until a maximal collection of non-intersecting curves, which we denote by $\base(\mu_X)$, is obtained. Additionally, to each curve $\alpha \in \base(\mu_X)$ we associate a transverse curve $\tau_\alpha$ by selecting from those curves that intersect $\alpha$ but no other curves in $\text{base}(\mu_X)$ a curve with shortest length. 

In \cite{Minsky2}, Minsky proved that for curves $\alpha$ and $\beta$, 
\begin{equation} \label{eq:intersection2} 
i(\alpha, \beta)^2 \leq \Ext_X(\alpha)\Ext_X(\beta),
\end{equation} 
where $i( \alpha, \beta)$ is the  geometric intersection number of $\alpha$ and $\beta$. 
So if $\Ext_X(\alpha)$ is sufficiently small, then every curve $\beta$ intersecting $\alpha$ must satisfy $\Ext_X(\beta) > \Ext_X(\alpha)$. Consequently, we have the following theorem.

\begin{theorem} \label{theorem:collar_lemma_Ext}
There exists a constant $\epsilon_0$ such that  for all $X \in \T(S)$, if a curve $\alpha$ satisfies $\Ext_X(\alpha) \leq \epsilon_0$, then $\alpha$ is in the base of every short marking on $X$.
\end{theorem}

Here and throughout the remainder of the paper, $\epsilon_0$ will denote the minimum of the constants in Theorems \ref{theorem:Ext_sum_mod},   \ref{theorem:collar_lemma_Ext}, and \ref{theorem:twisting_teich_geo}.

Given a subsurface $Y\subseteq S$, we let $\pi_Y(\mu_X)$ denote the usual subsurface projection of $\mu_X$ to $\C(Y)$ (see \cite{MMII}). Additionally, we define a projection map $\pi_Y \co \T(S) \rightarrow 2^{\C(Y)}$ by $X \mapsto \pi_Y(\mu_X)$, where $\mu_X$ is a choice of short marking on $X$. For a curve $\alpha$, we often write $\pi_\alpha$ instead of $\pi_{Y_\alpha}$.

Additionally, for each  curve $\alpha$, we define a map $\pi_{\Ho_\alpha} \co \T(S) \rightarrow 2^{\Ho_\alpha}$ as follows. Let $X \in \T(S)$. 
 If $\Ext_X(\alpha) > \epsilon_0$, define $n(X)=0$. Otherwise, define $n=n(X) \in \mathbb{Z}_{\geq 0}$ so that $\frac{\epsilon_0}{e^{n+1}} <\Ext_X(\alpha) \leq \frac{\epsilon_0}{e^{n}}$. We then define 
\[\pi_{\Ho_\alpha}(X)=\{(\tau,n(X)): \tau \in \pi_\alpha(\mu_X)\}.\]

For every subsurface $Y$, define $d_Y(\cdot ,\cdot )=\text{diam}_{\C(Y)}\pi_Y(\cdot ) \cup \pi_Y(\cdot)$ and similarly, for a curve $\alpha$ define $d_{\Ho_\alpha}(\cdot ,\cdot )=\text{diam}_{\Ho_\alpha}\pi_{\Ho_\alpha}(\cdot ) \cup \pi_{\Ho_\alpha}(\cdot)$.

The results in \cite{Durham}, \cite{EMR},\cite{MMII}, \cite{Rafi_comb} imply that taking the collection of subsurfaces of $S$ as an index set and 
\[\{\pi_Y\co\T(S) \rightarrow 2^{\C(Y)}:Y \text{ non-annular} \} \cup \{\pi_{\Ho_\alpha}\co\T(S) \rightarrow 2^{\Ho_\alpha} : \alpha \text{ a curve in } S \}\] as the collection of associated projection maps yields a hierarchically hyperbolic space (HHS)  structure on $\T(S)$. Throughout this paper, we will regard $\T(S)$ as an HHS space equipped with this structure. We will only need certain properties of this structure, described in the next subsection. See \cite{BHS} for the definition of an HHS structure.

\subsection{HHS boundary of Teichm\"uller space}
According to \cite{DHS}, the HHS structure on $\T(S)$ determines an associated boundary called the \emph{HHS boundary of $\T(S)$}, which we will denote $\partial \T(S)$. As a set, $\partial \T(S)$ is define to be 
\begin{multline}
\hspace{-10pt} \partial \T(S) =\Bigg\{\sum_{Y \subseteq S} c_Y\lambda_Y : \lambda_Y \in \partial \mathcal{C}(Y) \text{ for } Y \text{ non-annular, } \lambda_{Y_\alpha} \in \partial \Ho_\alpha \text{ for } Y_\alpha \text{ annular,} \Bigg. \\  \Bigg.  c_Y \geq 0, \text{ and }\sum_{Y \subseteq S} c_Y =1,  \text{ and if } c_{Y'}, c_{Y}>0, \text{ then } Y\text{ and } Y' \text{ are disjoint or equal} \Bigg\}. \nonumber
\end{multline}

Section 2 of \cite{DHS} describes a Hausdorff topology on $\T(S) \cup \partial \T(S)$ in which a sequence $(X_n)_{n \in \mathbb{N}}$ in $\T(S)$ converges to a point $\displaystyle{\sum \limits_{Y \subseteq S}  c_Y \lambda_Y}$ in $\partial \T(S)$, if and only if the following statements hold: 
Let $Y_1, \ldots, Y_k$ be the collection of subsurfaces with $c_Y >0$. 
\begin{enumerate}
\item For all $i=1,\ldots k$, if $Y_i$ is non-annular, then  $\displaystyle{\lim_{n \rightarrow \infty} \pi_{Y_i}(X_n) = \lambda_{Y_i}}$, and if $Y_i$ is annular with core curve $\alpha$, then $\lim \limits_{n\rightarrow \infty} \pi_{\Ho_{\alpha}}(X_n)=\lambda_{Y_i}$.
\item For all  subsurfaces $Y$, let $\ol{d_Y}$ denote $d_Y$ if $Y$ is non-annular and $d_{\Ho_\alpha}$ if $Y$ is annular with core curve $\alpha$.  Then for all subsurfaces $Y$ we have $\displaystyle{\lim_{n \rightarrow \infty} \frac{\ol{d_{Y}}(X_0, X_n)}{\sum \limits_{i=1}^k \ol{d_{Y_i}}(X_0, X_n)}=c_Y}.$
\end{enumerate}

\subsection{Continued fractions for irrational numbers}
Here we recall some elementary facts on continued fractions (see for example \cite{RS}). 
Let $\theta$ be an irrational number with continued fraction expansion $[a_0; a_1, a_2, a_3,\ldots ]$. That is, 
\begin{equation*}
\theta=a_0+\cfrac{1}{a_1+\cfrac{1}{a_2+\cfrac{1}{a_3+\cdots}}}
\end{equation*}
We will always assume $a_0 \geq 0$ and all other $a_n$ are strictly positive. 
We define the $n^{th}$ convergent of $\theta$ to be the reduced fraction $\frac{p_n}{q_n}=[a_0;a_1,a_2, \ldots, a_n]$. The numbers $p_n$ and $q_n$ are given recursively by

\begin{equation} \label{eq:recurrence}
q_n=a_n q_{n-1}+q_{n-2},  \hspace{10pt} q_{-1} =0, \text{ and } \hspace{10pt} q_{-2}=1 
\end{equation} 
and
\[p_n=a_n p_{n-1}+p_{n-2},  \hspace{10pt} p_{-1} =1, \text{ and } \hspace{10pt} p_{-2}=0 \]
 and satisfy
\begin{equation} \label{eq:denom_growth}
\frac{1}{q_n+q_{n+1}} \leq |p_n-\theta q_n| \leq \frac{1}{q_{n+1}} 
\end{equation} 
and
\begin{equation} \label{eq:det1}
|p_nq_{n+1}-q_np_{n+1}|=1.
\end{equation}
A simple but useful observation is that $\theta$ and each $p_n/q_n$ can be bounded as follows:
\begin{equation} \label{eq:Ext2}
 a_0 \leq \theta \leq a_0+1 \hspace{10pt} \text{ and } \hspace{10pt} a_0 \leq p_n/q_n \leq a_0+1.
\end{equation} 

\subsection{Teichm\"uller geodesic rays from irrational numbers}
Let $\theta_0, \theta_1, \theta_2$ be irrational numbers and $0<s<1$ and consider the corresponding Teichm\"uller geodesic ray $\G:[0, \infty) \rightarrow \T(S)$, described in Section \ref{sec:intro}, which we have parameterized by arc length.  We let $X_t$ denote $\mathcal{G}(t)$. 
For each $i$, let $Y_i$  denote the subsurface of $S$ that is the image of the slitted square $R_i$ under the gluing map, and let $\beta_i$ denote the boundary curve of $Y_i$.
Let  $p^i_n/q^i_n$ denote the $n^{th}$ convergent of $\theta_i$. Let $\alpha_i(n)$ denote the curve in $S$ corresponding to the trajectory in $R_i$ with slope $p^i_n/q^i_n$.
Define $T^i_n$ to be the balance time along $\mathcal{G}$ of $\alpha_i(n)$. We shall use these notations throughout the paper.    When we use them, it will be clear from context which irrational numbers and  Teichm\"uller geodesic ray we are working with. 

In \cite{Lenzhen}  Lenzhen gives an explicit formula for  $T^i_n$ and gives a useful  bound for the extremal length of $\alpha_i(n)$ along $\G$. 

\begin{theorem}[Lenzhen \cite{Lenzhen}, Lemma 1, proof of Lemma 3]  \label{theorem:lenzhen} For all $n \geq 0$  and $i=0,1,2$ 
\begin{enumerate}
\item $\Ext_t(\alpha_i(n)) \leq \left( \frac{\sqrt{1+\theta_i^2}}{\sqrt{1+\theta_i^2}-s|q^i_n\theta_i-p^i_n|} \right) \ell_{q_t}^2(\alpha_i(n))$ for $t \geq 0$.
\item The quadratic differential $q_t$ induces a flat structure on the torus $Y'_i$ obtained by ignoring the slit in $Y_i$.  In that metric, for all $t \in [T^i_n, T^i_{n+1}]$, a shortest curve in $Y_i'$ is $\alpha_i(n)$ or $\alpha_i(n+1)$. This statement also holds for the slitted torus $Y_i$ using the $q_t$-metric. Moreover, the length of $\alpha_i(n)$ in the metric $q_t$ induces on $Y_i'$ is equal to $\ell_{q_t}(\alpha_i(n))$. 
\item $T^i_n=\frac{1}{2} \log \frac{p^i_n\theta_i+q^i_n}{|q^i_n\theta_i-p^i_n|}$.
\end{enumerate}
\end{theorem}

\begin{remark} \label{remark:shortcurves}
{\rm There exists a constant $K$ such that given any unit area flat structure on a torus, there is a curve of length less than $K$. Thus, part 2 of Theorem \ref{theorem:lenzhen} tells us that for $t \in [T^i_{n}, T^i_{n+1}]$, we have $\ell_{q_t}(\alpha_i(n))$ or $\ell_{q_t}(\alpha(n+1))$, that is the length of the $q_t$-shortest curve in $Y_i$, is bounded uniformly above. }
\end{remark}
 Observe that $\beta_i$ is a closed leaf in the vertical foliation associated to $\G$. The following theorem of Choi, Rafi, and Series gives us useful information about how the projection of $\G$ to $\C(\beta_i)$ moves through $\C(\beta_i)$. 

\begin{theorem} [Theorem 5.13, \cite{CRS}] \label{theorem:twisting_teich_geo} There exists a constant $\epsilon_0$ depending only on $S$ such that the following holds. 
Let $\G$ be a Teichm\"uller geodesic with horizontal and vertical foliation $\nu^+$ and $\nu^-$, respectively. Suppose $\alpha$ is a closed leaf in $\nu^-$ and $\emph{\Ext}_t(\alpha) \leq \epsilon_0$. Then 
\[d_{\alpha}(\nu^+, X_t) \prec \frac{1}{\emph{\Hyp}_t(\alpha)}.\]
\end{theorem}

\section{Form of accumulation points of Teichm\"uller geodesics} \label{sec:form_acc_pt}

Throughout this section,  for each $i=0,1,2$ we fix sequences $(\theta_i(j))_{j=1}^\infty$ and $(n_i(j))_{j=1}^\infty$, where $\theta_i(j) \geq 2$ for all $i,j$. We then define \[\theta_i=[0;\underbrace{\theta_i(1), \ldots, \theta_i(1)}_{n_i(1)}, \ldots, \underbrace{\theta_i(j), \ldots, \theta_i(j)}_{n_i(j)}, \ldots ].\]
Additionally, we fix a slit length $s$. We let $S$ denote the genus $3$ surface and let $\G:[0,\infty) \rightarrow \T(S)$ denote the Teichm\"uller geodesic ray associated to $(\theta_0,\theta_1, \theta_2)$ with slit length $s$. 
Define $N_i(0)=0$ and for $k \geq 1$ define $\displaystyle{N_i(k)=\sum_{j=1}^kn_i(j)}$.

In this section, through a sequence of lemmas, we will show that if
the sequences $(n_i(j))_{j=1}^\infty$ grow sufficiently fast,  then there exists $\eta_i \in \partial \C(Y_i)$ such that every point in the limit set of $\G$ is of the form $\sum \limits_{i=0}^2 c_{Y_i} \eta_{i}$ for some  $c_{Y_i} \geq 0$. 

We begin with Lemma \ref{lemma:curvegraph_progress}, where we establish that the projection of $\G$ to $\C(Y_i)$ converges to a unique point $\eta_i \in  \partial \C(Y_i)$. 
This is almost immediate from the result of Rafi \cite{Rafi} that the projection of any Teichm\"uller geodesic to $\C(Y)$ is an unparameterized quasi-geodesic for every non-annular subsurface $Y$, but  we provide a direct proof in our setting that will also reveal some information about when $\G$ makes progress in $\C(Y_i)$ that will be useful later. From Lemma \ref{lemma:curvegraph_progress}, 
it will follow that 
 every point in the limit set of $\G$ is of the form 
\[\sum \limits_{i=0}^2 c_{Y_i}\eta_i +c_{\beta_i} \eta_{\beta_i},\]
where $\eta_{\beta_i}$ is the point in $\partial \Ho_{\beta_i}$. To determine what the constants $c_{Y_i}$ and $c_{\beta_i}$ can be, we must understand how fast the projection of $\G$ moves through each of the $\C(Y_i)$ and $\Ho_{\beta_i}$ for $i=0,1,2$ relative to one another. 
In Lemma \ref{lemma:curvegraph_progress}, we will see that $d_Y(X_0,X_{T^i_n}) \stackrel{+}{\asymp} n$.   From here,   we use Theorem \ref{theorem:lenzhen} to provide useful estimates for balance times $T^i_n$ (Lemma \ref{lemma:heavylifting}).  We will then prove Lemma \ref{lemma:horoproj}, which puts an upper bound on how fast the projection of $\G$ can move through a horoball $\Ho_{\beta_i}$.  We use this upper bound to prove that if  the sequence $(n_i(j))_{j=1}^\infty$ grows fast enough, then $c_{\beta_i}=0$ (Lemma \ref{lemma:small_horo_curve_graph_ratios}).

To simplify the notation, throughout the rest of this section, we will fix $i \in \{0,1,2\}$ and suppress $i$ in all the  associated notations. In particular, $Y, T_n, \theta(j), \beta, $ and $q_n$ will denote  $Y_i, T^i_n, \theta_i(j), \beta_i, $ and $q^i_n$, respectively. 

\

\begin{lemma} \label{lemma:curvegraph_progress}
 For all $n \geq 1$,  \[t \in [T_{n-1}, T_{n}] \hspace{5pt} \Rightarrow \hspace{5pt} d_{Y}(X_0, X_t)\stackrel{+}{\asymp} n.\]
Thus, the projection of $\mathcal{G}$ to $\C(Y)$ is an unparameterized quasi-geodesic converging to a unique point in $\eta_i \in \partial \C(Y)$.
\end{lemma}

\begin{proof}

Let $n \geq 1$. By (\ref{eq:det1}) the curves $\alpha(n-1)$ and $\alpha(n)$ are adjacent in $\C(Y)$. In fact, because the convergent $p_n/q_n$ has a depth $n$ continued fraction expansion with all but the zeroth coefficient at least 2, we have
 \begin{equation} \label{eq:curve_graph_dist}  d_{Y}(\alpha(0), \alpha(n))=n.
 \end{equation} 
 (See \cite{Series}). 
Fix  $t \in [T_{n-1}, T_{n}]$. By the triangle inequality, we have
 \[|d_Y(X_0, X_t) -n|=|d_Y(X_0, X_t)-d_Y(\alpha(0), \alpha(n))| \leq d_Y(X_0, \alpha(0))+d_Y(\alpha(n), X_t).\]
 We now show that  $d_Y(X_0, \alpha(0))$ and $d_Y(\alpha(n), X_t)$ are each bounded above by a constant depending only on $S$. 
 
 Consider the Euclidean cylinder $A$ in $S$ with core $\alpha(0)$ that is the union of the  $q_0$-geodesic representatives of $\alpha(0)$. Because $\theta, s \in (0,1)$, we have that $\Mod_0(A)  \geq \frac{1}{4}$. Thus, $\Ext_0(\alpha(0)) \leq 4$. Now for every curve $\beta$ in the base of the short marking $\mu_0$ on $X_0$, we also have that $\Ext_{0}(\beta)$ is bounded uniformly above by a constant depending only on $S$ (see \cite{Minsky4} and  Theorem 2.3 in \cite{Rafi}). So by inequality (\ref{eq:intersection2}),  the intersection of $\alpha(0)$ with every curve in  $\base(\mu_0)$ is bounded above uniformly. Therefore, $d_Y(X_0, \alpha(0))$ is bounded above uniformly.  
 
 Observe that (\ref{eq:denom_growth}) together with the fact that $\theta(j) \geq 2$ for all $j$ and $\theta, s \in (0,1)$ imply that $ \frac{\sqrt{1+\theta^2}}{\sqrt{1+\theta^2}-s|q_n\theta-p_n|} $ is bounded above by a uniform constant. So, Theorem \ref{theorem:lenzhen} tells us  
 \begin{equation}  \label{eq:ext<flat} 
 \Ext_{t}(\alpha(m)) \prec \ell^2_{q_t}(\alpha(m)) \hspace{7pt} \text{ for all } \hspace{5pt} m \geq 0.
 \end{equation} 
Combining this with Remark \ref{remark:shortcurves}, we find that $\Ext_{t}(\alpha(n-1))$ or $\Ext_{t}(\alpha(n))$ is bounded above uniformly.  An argument similar to that used above for $d_Y(X_0,\alpha(0))$ together with the fact that $d_Y(\alpha(n-1), \alpha(n))=1$ implies
$d_Y(\alpha(n), X_t)$ is bounded above by a uniform constant, as desired. 
 Therefore, \[d_Y(X_0, X_t) \stackrel{+}{\asymp} n.\]
Because this coarse equality is true for all $n$,  the projection of $\G$ to $\C(Y)$ is an unparameterized quasi-geodesic. Consequently, $\{\pi_Y(X_t)\}_{t \geq 0}$ accumulates on a unique point in $\partial \C(Y)$.

\end{proof}

\textbf{Remark:} Theorem \ref{theorem:lenzhen} of Lenzhen gives us an exact formula for $T_n$, but this formula is insufficient for our purposes because it requires us to know $\theta$ exactly. In Lemma \ref{lemma:heavylifting} we use Lenzhen's formula as a starting point to show that 
an initial segment of length $n+1$ of the continued fraction expansion of $\theta$ is all that is required to obtain a coarse estimate for $T_n$. 
We remark that Lenzhen, Modami, and Rafi \cite{LMR} also provided a coarse estimate with this property. The estimates we present in Lemma \ref{lemma:heavylifting} are more useful to us because the continued fractions we consider will have long stretches of the same number. 

\

Before stating the lemma, for $x \in \mathbb{R}$ we define 
\[\lambda(x)=\frac{x+\sqrt{x^2+4}}{2} \hspace{10pt} \text{ and } \hspace{10pt} \ol{\lambda}(x)=\frac{x-\sqrt{x^2+4}}{2} .\]

\begin{lemma} \label{lemma:heavylifting}
 There exists a uniform additive error such that for all $j \geq 1$ the following hold.  
\begin{enumerate}
\item For all $0 \leq \ell \leq n(j)$ we have \[\log q_{N(j-1)+\ell} \stackrel{+}{\asymp} \log q_{N(j-1)}+\ell\log \lambda(\theta(j)).\]
\item For all $ 0 \leq \ell \leq n(j)-1$, we have 
\[T_{N(j-1)+\ell} \stackrel{+}{\asymp} \log q_{N(j-1)}+(\ell+1/2)\log \lambda(\theta(j)).\]
\end{enumerate}
\end{lemma}
\begin{proof}  Fix  $j \geq 1$. 

\

\noindent \textbf{Proof of 1}: Equation (\ref{eq:recurrence}) says the $q_n$ are given recursively by 
\begin{equation} \label{eq:recursionqi} q_{N(j-1)+\ell}=\theta(j)q_{N(j-1)+\ell-1}+q_{N(j-1)+\ell-2} \hspace{10pt} \text{when} \hspace{5pt} 1\leq \ell \leq n(j).
\end{equation} 
The solution to this recursion is 
\[q_{N(j-1)+\ell}=A(j)\lambda(\theta(j))^{\ell}+B(j)\overline{\lambda}(\theta(j))^{\ell} \hspace{25pt} 0 \leq \ell \leq n(j),\]
where we define 
\begin{equation} \label{eq:AiBi} A(j)= \frac{q_{N(j-1)+1}-\overline{\lambda}(\theta(j))q_{N(j-1)}}{\lambda(\theta(j))-\overline{\lambda}(\theta(j))} \hspace{5pt} \text{ and } \hspace{5pt} B(j)= \frac{q_{N(j-1)} \lambda(\theta(j))-q_{N(j-1)+1}}{\lambda(\theta(j))-\overline{\lambda}(\theta(j))}.
\end{equation}

If $\ell=0$, statement 1 is clearly true. So  assume $1 \leq \ell \leq n(j)$. Observe that $-1 <\ol{\lambda}(\theta(j)) <0$ and  $\lambda(\theta(j)) >1$. This with equation (\ref{eq:recurrence}) and our assumption that $\theta(j) \geq 2$ implies

\begin{align*}
 \left|\frac{B(j)\overline{\lambda}(\theta(j))^{\ell}}{A(j)\lambda(\theta(j))^{\ell}}  \right|
 &=\left|\frac{q_{N(j-1)}\lambda(\theta(j))-q_{N(j-1)+1}}{q_{N(j-1)+1}-\overline{\lambda}(\theta(j))q_{N(j-1)}} \left(\frac{\overline{\lambda}(\theta(j))^{\ell}}{\lambda(\theta(j))^{\ell}}\right) \right| \\
& \leq \left|\frac{q_{N(j-1)}\lambda(\theta(j))-q_{N(j-1)+1}}{q_{N(j-1)+1}} \left(\frac{\overline{\lambda}(\theta(j))^{\ell}}{\lambda(\theta(j))^{\ell}}\right) \right| \\
& \leq \left|\frac{\overline{\lambda}(\theta(j))^\ell}{\lambda(\theta(j))^{\ell-1}} \right| + \left| \frac{\overline{\lambda}(\theta(j))^{\ell}}{\lambda(\theta(j))^{\ell}}\right| \\
& \leq 2 |\overline{\lambda}(\theta(j))| \hspace{5pt}  \leq  \hspace{5pt} 2|\overline{\lambda}(2)| .
\end{align*}
This implies that 
\begin{align} \label{eq:qni_est}
|\log q_{N(j-1)+\ell}- \log(A(j)\lambda(\theta(j))^{\ell})|&=
|\log[A(j)\lambda(\theta(j))^{\ell} +B(j)\overline{\lambda}(\theta(j))^{\ell}]- \log(A(j)\lambda(\theta(j))^{\ell})|  \nonumber\\ 
& =\left|\log\left(1+\frac{B(j)\overline{\lambda}(\theta(j))^{\ell}}{A(j)\lambda(\theta(j))^{\ell}} \right) \right| \nonumber\\
& \leq  |\log(1+2 \overline{\lambda}(2))|. 
\end{align}

To complete the proof of statement (1), we now show  $\log A(j) \stackrel{+}{\asymp} \log q_{N(j-1)}$.
It follows directly from  (\ref{eq:recursionqi}) and (\ref{eq:AiBi}) that 
\[\log A(j)  \leq \log \frac{2q_{N(j-1)+1}}{\lambda(\theta(j))} = \log  \frac{2(\theta(j) q_{N(j-1)}+q_{N(j-1)-1})}{\lambda(\theta(j))} \leq \log 4q_{N(j-1)},\]
and
\[\log A(j) \geq \log \frac{q_{N(j-1)+1}}{\lambda(\theta(j))-\ol{\lambda}(\theta(j))} \geq \log \frac{\theta(j) q_{N(j-1)}}{2\theta(j)} \stackrel{+}{\succ} \log q_{N(j-1)}.\]

\

\noindent \textbf{Proof of 2:} Let $n\geq 0$. 
Theorem \ref{theorem:lenzhen} (Lenzhen) tells us $T_n= \frac{1}{2}\log \frac{p_{n}\theta +q_{n}}{|q_{n}\theta -p_{n}|}$. We will use this  to first show that  $T_n$ is coarsely  $\frac{1}{2}\log q_nq_{n+1}$. 
We remark that Lenzhen, Modami, and Rafi \cite{LMR} obtain this same coarse estimate for the sequences they consider. Because our sequences do not fit their form, we derive the estimate for sequences in our setting. 

By   (\ref{eq:denom_growth}), we have

\[\frac{p_n\theta+q_n}{|q_n\theta-p_n|}  \geq (p_n\theta+q_n)q_{n+1} \geq  q_nq_{n+1},\]
and applying (\ref{eq:denom_growth}) and (\ref{eq:Ext2}) and the fact that $q_{n+1} >q_n$, we find
\[\frac{p_n\theta+q_n}{|q_n\theta-p_n|}  \leq (p_n+q_n)(q_n+q_{n+1}) 
 \leq (q_n)^2+q_nq_{n+1}+(q_n)^2+q_nq_{n+1} 
  \leq  4q_nq_{n+1}. \]
These inequalities show that
\begin{equation} \label{eq:Tni_est} T_n \stackrel{+}{\asymp} \frac{1}{2}\log q_{n}q_{n+1} \hspace{10pt} \text{for all  } n \geq 0.
\end{equation} 
 This together with statement 1 of the lemma implies that for  $0\leq \ell \leq n(j)-1$ 
\[T_{N(j-1)+\ell} \stackrel{+}{\asymp} \log q_{N(j-1)}+ (\ell+1/2) \log(\lambda(\theta(j))). \]

\end{proof}

The projection  of $\mathcal{G}$ to the horoball $\Ho_\beta$ depends on whether or not  the extremal length of $\beta$ is small. Thus, we now show that if  the slit length $s$ happens to be small enough, then the extremal length of $\beta$ is small at every time along $\G$. 

\begin{lemma} \label{lemma:ext_beta_small} 
If $s$ is sufficiently small, then $\Ext_t(\beta) \leq \frac{\epsilon_0}{e}$ for all $t \geq 0$.
\end{lemma}

\begin{proof} Let $t \geq 0$. 
 Because $\Ext_t\beta= \inf \frac{1}{\Mod_t A}$, where the infimum is taken over all annuli $A$ in $S$ with core $\beta$, to show $\Ext_t\beta$ is small, we must
exhibit such an annulus with large modulus. 
\begin{figure} 
\begin{center}
\begin{tikzpicture}
\coordinate  (n1) at ({cos(60) -sin(60)}  ,  {cos(60) +sin(60) } ) {};
\coordinate  (n2) at ({cos(60) +sin(60)} ,  {-cos(60) +sin(60)} ) {};
\coordinate  (n3) at ({-cos(60) +sin(60)} ,  {-cos(60) -sin(60)} ) {};
\coordinate  (n4) at ({-cos(60) -sin(60)} ,  {cos(60) -sin(60)} ) {};

\filldraw (0,0)[fill=gray!40!] circle (.99);
\filldraw (0,0)[fill=white]  circle (.6);

\node(A0) at (-1,-1.2) {$A_0$};

\draw [->] (-1.3,-1.1) to  [out=130,in=160](-.7,-.4); 

\draw (0,.6)--(0,-.6);

\draw  (n1)--(n2);
\draw (n2)--(n3);
\draw (n3)--(n4);
\draw (n4)--(n1);
\draw (n1)--(n2)--(n3)--(n4)--cycle;

\draw [->] (1.7,0) -- (4,0);

\node (arrowtext) at (2.85,.65) { $\left( {\begin{array}{cc}
   e^t & 0 \\       0 & e^{-t} \\      \end{array} } \right)$};

\coordinate   (m1)  at ({8+(e^1)*(cos(60) -sin(60)) }  ,  {(e^(-1))*(cos(60) +sin(60)) } ) {};
\coordinate   (m2) at ({8+(e^1)*(cos(60) +sin(60)) } ,  {e^(-1)*(-cos(60) +sin(60))} ) {};
\coordinate  (m3) at ({8+(e^1)*(-cos(60) +sin(60))} ,  {(e^(-1))*(-cos(60) -sin(60))} ) {};
\coordinate  (m4) at ({8+(e^1)*(-cos(60) -sin(60))} ,  {(e^(-1))*(cos(60) -sin(60))} ) {};

\coordinate (t1) at (8,({e^(-1))*(.6)} ) {};
\coordinate (t2) at (8,({e^(-1))*(-.6)} ) {};

\filldraw (8,0)[fill=gray!40!] circle (.41);
\filldraw (8,0)[fill=white] circle ({e^(-1)*(.6)});

\node(At) at (8,-.9) {$A_t$};

\draw [->] (8.3,-.9) to  [out=20,in=350](8.27,-.1); 
\draw (t1)--(t2);

\draw (m1)--(m2);
\draw (m2)--(m3);
\draw (m3)--(m4);
\draw (m4)--(m1);
\end{tikzpicture}
\caption{Annulus $A_t$ in $Y$ with core curve $\beta_i$ and large modulus in $X_t \in \T(S)$.}
\label{fig:Modann}
\end{center}
\end{figure}
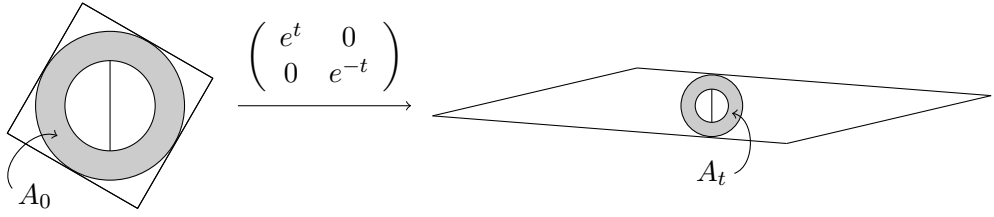

Let $A_t$ be the annulus contained in $Y$ with core curve $\beta$ and boundary components Euclidean circles in the flat $q_t$-metric as pictured  in Figure \ref{fig:Modann}. Let $r(t)$ and $R(t)$ denote the flat $q_t$-length of the inner and outer radii of $A_t$ respectively. As we move along $\G$, the flat $q_t$-length of a segment in $S$ shrinks at most exponentially. Thus, $R(t) \geq e^{-t}R(0)$.
Observe that 
\[\Mod A_t= \frac{1}{2\pi} \log \frac{R(t)}{r(t)} \geq \frac{1}{2\pi} \log \frac{e^{-t}R(0)}{\frac{1}{2} se^{-t}}=\frac{1}{2\pi} \left(\log 2R(0) +\log \frac{1}{s}\right).\]
Therefore,  provided that $s$ is sufficiently small, we have $\Ext_t(\beta) \leq \frac{\epsilon_0}{e}$. 
\end{proof}

Notice that how small $s$ must be for the conclusion of Lemma \ref{lemma:ext_beta_small} to hold is independent of $\theta$. Thus, throughout the remainder of this paper, we can and do assume the slit length $s$ is small enough to satisfy Lemma \ref{lemma:ext_beta_small}.

\begin{lemma} \label{lemma:horoproj} 
For all  $t \geq 1$ we have
\[d_{\Ho_{\beta}}(X_0, X_{t}) \stackrel{+}{\prec} \log t,\]
where the error constant depends only on $s$. 
\end{lemma}

\begin{proof}
For each $t \geq 0$, define $h(t)$ and $v(t)$ so that $\pi_{\Ho_{\beta}}(X_t)=(h(t), v(t)).$ 
Recall that two vertices at height $n$ in $\Ho_{\beta}$ are adjacent if their horizontal components are within $e^n$ of each other in $\C(\beta)$.   
By Lemma \ref{lemma:ext_beta_small}, $\Ext_t(\beta) \leq \frac{\epsilon_0}{e}$ for all $t \geq 0$. So the construction of $\Ho_{\beta}$ implies $v(t) \geq 1$ and $v(t) \stackrel{+}{\asymp} \log \frac{1}{\Ext_t \beta}$ for all $t \geq 0$. These observations together with the triangle inequality imply
\begin{align*}
d_{\Ho_{\beta}}(X_0, X_{t}) \stackrel{+}{\prec} v(0)+\log \left(\frac{1}{\Ext_t\beta}\right)+ \frac{d_{\beta}(X_0, X_t)}{ e^{v(t)}-1}.
\end{align*}
To establish the desired bound on $d_{\Ho_{\beta}}(X_0, X_{t})$, our strategy is to prove the following claims, which for now we assume are true. 

\textbf{Claim 1:}  $\frac{1}{\Ext_{t}(\beta)} \prec  t$ for $t \geq 0$.

\textbf{Claim 2:}  $d_{\beta}(X_0 X_t) \prec \frac{1}{\Ext_t(\beta)}$ for  $t \geq 0$.

Claim 1 implies that $v(0)$ is bounded above uniformly and 
if $t \geq 1$, it implies that  $\log \left(\frac{1}{\Ext_t\beta}\right) \stackrel{+}{\prec} \log t$ (this is because $t$ is bounded uniformly away from 0).  
Claim 2 together with the fact that $\Ext_t(\beta) \leq \frac{\epsilon_0}{e}$ implies that 
 $\frac{d_{\beta}(X_0, X_t)}{e^{v(t)}-1 }$ is bounded above by a uniform constant.   Combining these observations, we have \[d_{\Ho_{\beta}}(X_0, X_t) \stackrel{+}{\prec} \log t \hspace{7pt} \text{ for all } \hspace{5pt}  t \geq 1.\]
Thus, all that remains is to prove the claims. 

\

\noindent \textbf{Proof of Claim  1:} Let $t \geq 0$. We consider the flat structure determined by $q_t$. The flat annulus with core $\beta$ is degenerate. 
Observe that $\ell_{q_t}(\beta)=2se^{-t}$ and that the distance between the boundary components of the expanding annulus in the direction opposite $Y$ is at most $\frac{1}{2}se^{-t}$. So by Theorem \ref{theorem:modulus},  the modulus of this expanding annulus is uniformly bounded  above. It then follows by Theorems \ref{theorem:Ext_sum_mod} and  \ref{theorem:modulus} that 
\[\frac{1}{\Ext_t(\beta)} \asymp \log \frac{d_t}{\ell_{q_t}(\beta)} =t+\log \frac{d_t}{2s},\]
where $d_t$ is the $q_t$-distance between the boundary components of the expanding annulus in the direction of $Y$ at time $t$. Now $d_t$ is at most half the length of the shortest $q_t$-length curve in $Y$ at time $t$, which is bounded above uniformly (see Remark \ref{remark:shortcurves}). 
So, we have $\frac{1}{\Ext_t(\beta)} \prec t ,$
establishing Claim 1. 
\

\noindent \textbf{Proof of Claim 2:} For $t \geq 0$
let $\mu_t$ be a short marking on $X_t$. 
Because $\Ext_t (\beta) \leq \epsilon_0$, 
we know  $\beta \in \base(\mu_t) $. This tells us $\pi_{\beta}(X_t)$ is the projection to $\C(\beta)$ of the transversal in $\mu_t$ associated to $\beta$. 
Let $\nu^+$ denote the horizontal foliation of $\G$. Because $\beta$ is a leaf of the vertical foliation of $\G$, by Theorem \ref{theorem:twisting_teich_geo}
\begin{equation} \label{eq:twist} d_{\beta}(\nu^+, X_t) \prec \frac{1}{\Hyp_t(\beta)}.
\end{equation}
Further observe that because $\Ext_t(\beta) \leq \epsilon_0$, Theorem \ref{theorem:maskit} tells us that $\frac{1}{\Hyp_t(\beta)} \stackrel{*}{\asymp} \frac{1}{\Ext_t(\beta)}$. 
So (\ref{eq:twist}) and  Claim 1 imply that
\[ d_{\beta}(X_0, X_t) \leq d_{\beta}(\nu^+,X_0)+d_{\beta}(\nu^+, X_t)  
\prec \frac{1}{\Ext_0(\beta)}+\frac{1}{\Ext_t(\beta)} \prec \frac{1}{\Ext_t(\beta)},\]
proving Claim 2 and thus  completing the proof of the lemma. 
\end{proof}

\begin{convention} \label{conv:fast} {\rm Throughout the rest of this paper, when we say the sequence $(n_i(j))_{j=1}^\infty$ \textit{grows sufficiently fast}  we shall mean that for each $k$ we have $n_i(k)$ is larger than some function of the numbers in  $(n_\ell(j))_{j=1}^{k-1}$ and $(\theta_\ell(j))_{j=1}^{k+1}$ for each $\ell\in \{0,1,2\}$,  where the function varies based on the context in which this phrase is used. }
\end{convention} 
\begin{lemma} \label{lemma:small_horo_curve_graph_ratios}
If the sequence $(n(j))_{j=1}^\infty$ grows sufficiently fast, then  
\[\frac{d_{\Ho_{\beta}}(X_{0},X_t)}{d_{Y}(X_{0}, X_t)} \rightarrow 0 \hspace{7pt} \text{as} \hspace{7pt} t\rightarrow \infty. \]
\end{lemma}

\begin{proof}
Consider $t \geq T_{1} \geq 1$.  For some  $k \geq 1$ and $0 \leq \ell \leq n(k)-1$ we have that 
\[ T_{N(k-1)+\ell-1} < t \leq T_{N(k-1)+\ell}.\]
Regardless of how fast the $n(j)$ are growing, the following will be true. 
 Lemmas \ref{lemma:horoproj} and \ref{lemma:heavylifting} imply that 
\begin{align} \label{eq:horo_UB}
d_{\Ho_{\beta}}(X_0, X_t) \stackrel{+}{\prec} \log t &\leq \log T_{N(k-1)+\ell}  \nonumber\\
& \stackrel{+}{\prec} \log\left( \log q_{N(k-1)} +\left(\ell+\frac{1}{2}\right) \log \lambda(\theta(k))\right) \nonumber  \\
& \stackrel{+}{\prec} \log \left( \log q_{N(k-2)}+n(k-1)\log \lambda(\theta(k-1))+\left(\ell+\frac{1}{2}\right) \log \lambda(\theta(k))\right), 
\end{align}
and Lemma \ref{lemma:curvegraph_progress} implies that 
\begin{equation} \label{eq:Y_LB} 
d_{Y}(X_0, X_t) \stackrel{+}{\succ} N(k-1)+\ell =N(k-2) +n(k-1)+\ell.
\end{equation} 

Observe that $N(k-2)$, $q_{N(k-2)}$, $\lambda(\theta(k-1))$, and $\lambda(\theta(k))$ are completely determined by $(\theta(j))_{j=1}^k$ and $(n(j))_{j=1}^{k-2}$, and thus are completely independent of $n(k-1)$. Further observe that if $n(k-1)$ is sufficiently large 
relative to the numbers in $(\theta(j))_{j=1}^k$ and $(n(j))_{j=1}^{k-2}$, then the ratio of the upper bound of (\ref{eq:horo_UB}) to the lower bound of (\ref{eq:Y_LB}) is arbitrarily small, implying that $\frac{d_{\Ho_{\beta}}(X_0,X_t)}{d_{Y}(X_0, X_t)} $ is also small. 
This proves the lemma. 
\end{proof}

\section{Teichm\"uller geodesics with exotic limit sets} \label{sec:main_result}
Throughout this section, we fix $s$ sufficiently small in the sense of Lemma \ref{lemma:ext_beta_small}.
We fix a continuous map $\gamma \co\mathbb{R} \rightarrow \simp$ to the standard 2-simplex, and let $\gamma_i$ denote the $i^{th}$ component function of $\gamma$. In this section, we will show how to carefully choose infinite sequences $(\theta_i(j))_{j = 1}^{\infty}$ and $(n_i(j))_{j =1}^{\infty}$ for $i=0,1, 2$ and an embedding  $\simp \rightarrow \partial \T(S)$  so that the limit set of the associated Teichm\"uller geodesic  ray is homeomorphic to the image of  $\ol{\gamma(\mathbb{R})}$, proving Theorem \ref{theorem:main}.

 We also fix  a sequence $(t_j)_{j=1}^\infty $ in $\mathbb{R}$ so that $(\gamma(t_j))_{j=1}^\infty$ is dense in $\gamma(\mathbb{R})$ and
 \begin{equation} \label{eq:fine_seq} |\gamma_i(t_{j-1})-\gamma_i(t_j)|<\epsilon_j  \hspace{5pt} \text{ for each } \hspace{5pt}  i=0,1,2 \hspace{5pt} \text{ and } j \geq 2,
 \end{equation} 
where $(\epsilon_j)$  is some decreasing sequence, $\epsilon_j <\frac{1}{2}$, and $\lim \limits_{j \rightarrow \infty} \epsilon_j =0$. 
 
Let $L$ denote the additive error in the coarse estimates of Lemma \ref{lemma:heavylifting}.  Now for $i=0,1,2$ choose the sequence $(\theta_i(j))_{j=1}^\infty$ so that the following hold for all $j \geq 1$ and $i,\ell=0,1, 2$:

\begin{equation} \label{eq:theta_i_large_rel_L}
\log \lambda(\theta_i(j)) \geq 4L, 
\end{equation}
\begin{equation} \label{eq:shrinking_ratios}
\frac{\log \lambda(\theta_\ell(j))}{\log \lambda(\theta_i(j+1))}<\epsilon_{j+1}, 
\end{equation}
 and 
 \begin{equation} \label{eq:curve_est} 
\left|\frac{[\log \lambda(\theta_i(j))]^{-1}}{\sum \limits_{ \ell=0}^2 [\log \lambda(\theta_\ell(j))]^{-1}} -\gamma_i(t_{j})\right| < \epsilon_{j+1}.
\end{equation}

Given a sequence  $(n_i(j))_{j=1}^k$, $i=0,1,2$, we define $N_i(0)=0$ and $N_i(k)=\sum \limits_{j=1}^k n_i(j)$ for $k \geq 1$. When we say the Teichm\"uller geodesic ray corresponding to $(n_i(j))_{j=1}^\infty$, $i=0,1,2$, we shall mean the Teichm\"uller geodesic ray $\G$ with slit length $s$ associated to $(\theta_0, \theta_1, \theta_2)$, where 
\[\theta_i=[0; \underbrace{\theta_i(1), \ldots, \theta_i(1)}_{n_i(1)}, \ldots, \underbrace{\theta_i(j), \ldots, \theta_i(j)}_{n_i(j)}, \ldots].\]
We then use $\G$ to define a map $\phi \co [0,\infty) \rightarrow \bigtriangleup^2$ given by
\[ t\hspace{5pt}  \mapsto \hspace{5pt} \frac{1}{\sum \limits_{i=0}^2d_{Y_i}(X_0,X_t)}(d_{Y_0}(X_0,X_t), d_{Y_1}(X_0, X_t), d_{Y_2}(X_0, X_t)),\]
where as usual $X_t$ is the point on $\G$ distance $t$ from the base of the ray. 
We will write $\phi_i$ to denote the $i^{th}$ component of $\phi$. 

We will show how to pick the $(n_i(j))_{j=1}^\infty$ so that if $t$ is between the balance times $T^0_{N_0(k-1)-1}$ and $T^0_{N_0(k)-1}$, then for each $i$ we have $\phi_i(t)$ is close to $\frac{[\log \lambda(\theta_i(k-1))]^{-1}}{\sum \limits_{ \ell=0}^2 [\log \lambda(\theta_\ell(k-1))]^{-1}}$ or $\frac{[\log \lambda(\theta_i(k))]^{-1}}{\sum \limits_{ \ell=0}^2 [\log \lambda(\theta_\ell(k))]^{-1}}$ and thus close to $\gamma_i(t_{k-1})$ or $\gamma_i(t_k)$, which are close to each other by (\ref{eq:fine_seq}). 
 As a first step toward this goal, we prove the following lemma. (See Convention \ref{conv:fast} for our definition of sufficiently fast growth.)

\begin{lemma}\label{lemma:bound_comp_phi} For each $i=0,1,2$ suppose  $(n_i(j))_{j=1}^{\infty}$  grows sufficiently fast and that  the balance times along the  corresponding Teichm\"uller geodesic ray satisfy \[T^i_{N_i(k)-3} < T^0_{N_0(k)-1} \leq T^i_{N_i(k)-1} \hspace{20pt} \text{ for all } \hspace{4pt} k\geq 1 \hspace{4pt} \text{ and } \hspace{4pt} i=0,1,2.\] 
Then for all $k \geq 2$ and $i=0,1,2$
\begin{equation} \label{eq:bound_comp_phi}  t \in [T^0_{N_0(k-1)-1}, T^0_{N_0(k)-1}) \hspace{10pt}  \Longrightarrow  \hspace{10pt} |\phi_i(t)-\gamma_i(t_k)| \leq 11\epsilon_k.
\end{equation}
\end{lemma}

\begin{proof}
Let $\mathcal{G} \co [0,\infty) \rightarrow \T(S)$ be the Teichm\"uller geodesic  associated to $(n_i(j))_{j=1}^\infty, i=0,1,2$.
Fix $k \geq 2$ and $t \in [T^0_{N_0(k-1)-1}, T^0_{N_0(k)-1})$. For each $i=0,1,2$, define $ m(i)$ so that 
\begin{equation} \label{eq:def_m(i)} T^i_{N_i(k-1)+m(i)-1} < t \leq T^i_{N_i(k-1) +m(i)}.
\end{equation} 
Fix $\ell \in \{0,1,2\}$. Then  Lemma \ref{lemma:curvegraph_progress} implies that $\phi_\ell(t)$ is bounded above and below as follows:
\begin{equation} \label{eq:phi_bounds}
 \frac{N_\ell(k-1)+m(\ell)-R}{\sum \limits_{i=0}^2 (N_i(k-1)+m(i) +R)}  \leq \frac{d_{Y_\ell}(X_0, X_t)}{\sum \limits_{i=0}^2 d_{Y_i}(X_0, X_t)}  \leq  \frac{N_\ell(k-1)+m(\ell)+R}{\sum \limits_{i=0}^2 (N_i(k-1)+m(i) -R)}, 
\end{equation} 
where $R>0$ denotes the additive error from Lemma \ref{lemma:curvegraph_progress}.
Thus, to bound $\phi_\ell(t)$ we must compare $m(\ell)$ to   $m(i)$ for each $i=0,1,2$.

First, observe  our assumption that
\[T^i_{N_i(k-1)-3} < T^0_{N_0(k-1)-1} \hspace{10pt} \text{ and } \hspace{10pt} T^0_{N_0(k)-1} \leq T^i_{N_i(k)-1},\]
 implies  $-2 \leq m(i) \leq n_i(k)-1$. This means that  $m(i)+ 2\geq 0$.
Now by the definition of the  $m(i)$, we know that 
\begin{equation}  \label{eq:times}
T^\ell_{N_\ell(k-1)+m(\ell)-1} \leq T^i_{N_i(k-1)+m(i)} \leq T^i_{N_i(k-1)+m(i)+2}.
\end{equation}
So provided that $m(\ell) -1\geq 0$,  (\ref{eq:times}) together with Lemma \ref{lemma:heavylifting} part 2 implies that
\begin{equation} \label{eq:mimj_time_ineq} \log q^\ell_{N_\ell(k-1)} +\left(m(\ell)-\frac{1}{2}\right) \log \lambda(\theta_\ell(k)) -L \leq \log q^i_{N_i(k-1)} +\left(m(i) +\frac{5}{2}\right)\log\lambda(\theta_i(k)) +L.
\end{equation}
Now  (\ref{eq:mimj_time_ineq}), Lemma \ref{lemma:heavylifting} part 1, and (\ref{eq:shrinking_ratios})  imply 
\begin{align} \label{eq:m(i)} 
m(i) &\geq -\frac{5}{2} -\frac{ \log q^i_{N_i(k-1)} +2L +\frac{1}{2}\log\lambda (\theta_\ell(k))}{\log \lambda(\theta_i(k))} +\frac{\log\lambda(\theta_\ell(k))}{\log \lambda(\theta_i(k))} m(\ell)  \nonumber\\
& \geq -\frac{5}{2}- \frac{ \log q^i_{N_i(k-2)} +3L +\frac{1}{2}\log\lambda (\theta_\ell(k))}{\log \lambda(\theta_i(k))} -\epsilon_kN_i(k-1) +\frac{\log\lambda(\theta_\ell(k))}{\log \lambda(\theta_i(k))} m(\ell) \nonumber\\
&=-H_{i,\ell}(k)-\epsilon_kN_i(k-1) +\frac{\log\lambda(\theta_\ell(k))}{\log \lambda(\theta_i(k))} (m(\ell)+2),
\end{align}
where $H_{i,\ell}(k)$ is defined precisely so that the equality holds. Notice that $H_{i,\ell}(k)$ is completely determined by  the finite sequences $(n_i(j))_{j=1}^{k-2}$ and $(\theta_i(j))_{j=1}^k$ and $\theta_\ell(k)$. 
Further observe that the lower bounds given above for $m(i)$ still holds even if $m(\ell) \leq 0$ since  $m(i) \geq -2$ and $\log q^i_{N_i(k-2)}$, $L$, $\log \lambda (\theta_i(k))$, and $\log \lambda(\theta_\ell(k))$ are all greater than $0$.

We assume $n_i(j) \geq 3$ for all $i,j$ (as part of our sufficiently fast growth assumption). Lemma \ref{lemma:heavylifting} part 2 and our assumption on the balance times tell us 
\begin{equation} \label{eq:ratio_bounds}\frac{N_i(k-1)+R}{T^0_{N_0(k-1)-1}} \leq \frac{N_i(k-1)+R}{T^i_{N_i(k-1)-3}} \leq \frac{N_i(k-2)+n_i(k-1)+R}{\log q^i_{N_i(k-2)} +(n_i(k-1)-\frac{5}{2})\log \lambda(\theta_i(k-1))-L}.
\end{equation}
and
\begin{align} \label{eq:ratio_bounds_2} \frac{N_i(k-1)-2H_{i,\ell}(k)-2R}{T^0_{N_0(k-1)-1}}  &\geq \frac{N_i(k-1)-2H_{i,\ell}(k)-2R}{T^i_{N_i(k-1)-1}} \nonumber \\
&\geq \frac{N_i(k-2)+n_i(k-1)-2H_{i,\ell}(k)-2R}{\log q^i_{N_i(k-2)} +(n_i(k-1)-\frac{1}{2})\log \lambda(\theta_i(k-1))+L}. 
\end{align}
If $n_i(k-1)$ is sufficiently large, the right hand sides  of  inequalities (\ref{eq:ratio_bounds}) and (\ref{eq:ratio_bounds_2})  are both greater than $0$ and  very close to $1/\log \lambda(\theta_i(k-1))$. 
The crucial part of this proof is to notice that how large $n_i(k-1)$ must be to guarantee a certain prescribed closeness is determinable from $\theta_\ell(k)$ and the numbers in the finite sequences $(n_i(j))_{j=1}^{k-2}$ and $(\theta_i(j))_{j=1}^k$. 
Thus, if  $(n_i(j))_{j=1}^\infty$ grows sufficiently fast for each $i=0,1,2$, then 
\[0<\frac{N_\ell(k-1) +R }{\sum \limits_{i=0}^2 \left(N_i(k-1) -2H_{i,\ell}(k)-2R \right)}\leq \frac{ [\log \lambda(\theta_\ell(k-1))]^{-1}}{ \sum \limits_{i=0}^2 [\log \lambda(\theta_i(k-1))]^{-1}}+\epsilon_{k},\]
which together with (\ref{eq:m(i)}),  (\ref{eq:curve_est}),  and (\ref{eq:fine_seq})  implies that 
\begin{align*}
\frac{d_{Y_\ell}(X_0, X_t)}{\sum \limits_{i=0}^2 d_{Y_i}(X_0, X_t)} & \leq \frac{N_\ell(k-1)+m(\ell)+R}{\sum \limits_{i=0}^2 (N_i(k-1)+m(i) -R)} \\
&\leq \left( \frac{1}{1-\epsilon_k} \right) \frac{N_\ell(k-1) +R +(m(\ell)+2)}{\sum \limits_{i=0}^2 \left( N_i(k-1) -2H_{i,\ell}(k)-2R \right)+\left(\sum \limits_{i=0}^2 \frac{\log \lambda(\theta_\ell(k))}{\log \lambda(\theta_i(k))} \right)(m(\ell)+2)}  \\
& \leq \left( \frac{1}{1-\epsilon_k} \right) \max  \left\{ \frac{N_\ell(k-1) +R }{\sum \limits_{i=0}^2 \left(N_i(k-1) -2H_{i,\ell}(k)-2R \right)}, \frac{ [\log \lambda(\theta_\ell(k))]^{-1}}{ \sum \limits_{i=0}^2 [\log \lambda(\theta_i(k))]^{-1}}
\right\} \\
& \leq  \left( \frac{1}{1-\epsilon_k} \right) \max  \left\{ \frac{ [\log \lambda(\theta_\ell(k-1))]^{-1}}{ \sum \limits_{i=0}^2 [\log \lambda(\theta_i(k-1))]^{-1}}+\epsilon_{k}, \frac{ [\log \lambda(\theta_\ell(k))]^{-1}}{ \sum \limits_{i=0}^2 [\log \lambda(\theta_i(k))]^{-1}}
\right\}  \\
&\leq \left( \frac{1}{1-\epsilon_k}\right)(\gamma_\ell(t_k)+3\epsilon_k) \leq \gamma_\ell(t_k) +11\epsilon_k.
\end{align*}
 Note that the last inequality follows because  $\gamma(t_{k}) \in \simp$ implies $\gamma_\ell(t_{k})\leq 1$, and   $\epsilon_k <\frac{1}{2}$ implies  $\frac{1}{1-\epsilon_k} \leq 1 +2\epsilon_k$.

 A similar argument shows that if $(n_i(j))_{j=1}^\infty$ grows sufficiently fast for all  $i=0,1,2$, then
\[\frac{d_{Y_\ell}(X_0, X_t)}{\sum \limits_{i=0}^2 d_{Y_i}(X_0, X_t)} \geq \gamma_\ell(t) - 11\epsilon_k.\]
\end{proof}

The goal of the next lemma is to show that sequences satisfying the hypotheses of Lemma \ref{lemma:bound_comp_phi} can actually be constructed.

\begin{lemma} \label{lemma:balance_time_orders}
Let $k \geq 1$. Given $(n_i(j))_{j=1}^{k-1}$, $i=0,1,2$ and any number $N$, for each $i$ there exists  $n_i(k) \geq N$  so that the following holds. If $\theta_i$ is any irrational number whose continued fraction expansion begins with 
\[0, \underbrace{\theta_i(1), \ldots, \theta_i(1)}_{n_i(1)}, \ldots ,\underbrace{\theta_i(k), \ldots ,\theta_i(k)}_{n_i(k)}\]
for each $i$,
then the balance times on the Teichm\"uller geodesic ray  corresponding to $(\theta_0, \theta_1, \theta_2)$ satisfy
\[T^i_{N_i(k)-3} < T^0_{N_0(k)-1} \leq T^i_{N_i(k)-1} \hspace{20pt} \text{for } \hspace{4pt} i=0,1,2.\]
\end{lemma}

\begin{proof}
 Suppose we are given $(n_i(j))_{j=1}^{k-1}$ for each $i$ and a number $N$, which we may assume is at least $3$. Choose $n_0(k) \geq N$ so that for each $i=1,2$ if $\ell$ is an integer satisfying 
 
\begin{equation} \label{eq:pre_time_est} \log q^0_{N_0(k-1)}+\left(n_0(k)-\frac{1}{2}\right)\log \lambda(\theta_0(k))+L \leq \log q^i_{N_i(k-1)}+\left(\ell-\frac{1}{2}\right)\log \lambda(\theta_i(k))-L,
\end{equation}
then $\ell \geq N$. 
 Now for $i=1,2$ choose $n_i(k)$ to be the smallest integer $\ell$ satisfying inequality (\ref{eq:pre_time_est}). Then we have $n_i(k) \geq N \geq 3$ for all $i=0,1,2$. 
 
For each  $i=0,1,2$,  let $\theta_i$ be any irrational number whose continued fraction expansion begins with 
\[0, \underbrace{\theta_i(1), \ldots, \theta_i(1)}_{n_i(1)}, \ldots, \underbrace{\theta_i(k), \ldots, \theta_i(k)}_{n_i(k)}.\]
Consider the Teichm\"uller geodesic ray corresponding to $(\theta_0, \theta_1, \theta_2)$. Lemma \ref{lemma:heavylifting} part 2 and (\ref{eq:pre_time_est}) tells us that $T^0_{N_0(k)-1} \leq T^i_{N_i(k)-1}$ for each $i$.

We now show that because we chose $\theta_i(k)$ to be large relative to $L$, then necessarily $T^i_{N_i(k)-3} < T^0_{N_0(k)-1}$.
Observe that 
\begin{align*}
T^i_{N_i(k)-3} &\leq \log q^i_{N_i(k-1)}+\left(n_i(k)-\frac{5}{2}\right)\log \lambda(\theta_i(k))+L & \text{by Lemma } \ref{lemma:heavylifting} \text{ part }2\\
&\leq  \log q^i_{N_i(k-1)}+\left(n_i(k)-\frac{3}{2}\right)\log \lambda(\theta_i(k))-3L & \text{by Eq.(\ref{eq:theta_i_large_rel_L})}\\
& < \log q^0_{N_0(k-1)}+\left(n_0(k)-\frac{1}{2}\right)\log \lambda(\theta_0(k))-L &\text{by def. of } n_i(k) \\
& \leq T^0_{N_0(k)-1} & \text{by Lemma } \ref{lemma:heavylifting} \text{ part }2.
\end{align*}
\end{proof}

We now use Lemmas \ref{lemma:bound_comp_phi} and \ref{lemma:balance_time_orders} to prove our main result, Theorem \ref{theorem:main}, which we rephrase as Theorem \ref{theorem:main_rephrased} below. 

\begin{theorem} \label{theorem:main_rephrased}
There exists a triple of irrational numbers such that the limit set in $\partial \T(S)$ of the associated Teichm\"uller geodesic ray  is
\[\{c_1\eta_1+c_2\eta_2+c_3\eta_3:(c_1,c_2,c_3) \in \ol{\gamma(\mathbb{R})}\},\]
for some $\eta_i \in \partial \C(Y_i)$, $i=0,1,2$. 
\end{theorem}

\begin{proof}

By Lemma \ref{lemma:balance_time_orders}, we can choose sequences $(n_i(j))_{j=1}^\infty$ growing sufficiently fast in  the sense of both Lemmas  \ref{lemma:small_horo_curve_graph_ratios} and \ref{lemma:bound_comp_phi} such that the corresponding geodesic ray $\mathcal{G}: [0, \infty) \rightarrow \T(S)$ satisfies
\[T^i_{N_i(k)-3} <T^0_{N_0(k)-1} \leq T^i_{N_i(k)-1} \hspace{5pt} \text{ for all  } \hspace{5pt} k \geq 1.\] 
(See Convention \ref{conv:fast} for our definition of sufficiently fast growth.)
We can now apply Lemmas \ref{lemma:small_horo_curve_graph_ratios} and \ref{lemma:bound_comp_phi}  to conclude that for each $i=0,1,2$  and $k \geq 2$

\begin{equation} \label{eq:horoprojbound} 
\lim_{t\rightarrow \infty} \frac{d_{\Ho_{\beta_i}}(X_{0}, X_t)}{d_{Y_i}(X_{0}, X_t)} =0
\end{equation}
and 
\begin{equation} \label{eq:phicompbound}
 t \in [T^0_{N_0(k-1)-1}, T^0_{N_0(k)-1}) \hspace{5pt} \Longrightarrow \hspace{5pt} |\phi_i(t)-\gamma_i(t_k)| \leq 11\epsilon_k.
\end{equation} 

Let $\cL$ denote the limit set of $\G$ in $\partial \T(S)$.  
By the definition of the topology of $\T(S)\cup \partial \T(S)$, equation (\ref{eq:horoprojbound})  and  Lemma \ref{lemma:curvegraph_progress} imply that for some $\eta_i \in \partial \C(Y_i)$
\[\cL=\{ c_1\eta_1+c_2\eta_2+c_3\eta_3:(c_1, c_2, c_3) \in \cL_\phi \},\]
where $\cL_\phi$ denotes the set of accumulation points of $\phi$ in $\simp$. 
So, to complete this proof, we must show that $\cL_\phi=\overline{\gamma(\R)}$. 

Throughout the rest of the proof, we think of $\simp$ as a subset of $\mathbb{R}^3$ equipped with the $\ell_1$-norm. 
Consider a point $P \in \overline{\gamma(\R)}$. Because $(\gamma(t_j))_{j=1}^\infty$ is a sequence dense in $\gamma(\R)$,  some subsequence $\gamma(t_{j_n}) \rightarrow P$ as $n\rightarrow \infty$. Observe that (\ref{eq:phicompbound}) tells us that $\phi(T^0_{N_0(j_n-1)-1})$ is within $33 \epsilon_{j_n}$ of $\gamma(t_{j_n})$. Since $\epsilon_{j_n}\rightarrow 0$ as $n\rightarrow \infty$, it must be that $\phi(T^0_{N_0(j_n-1)-1}) \rightarrow P$ as $n\rightarrow \infty$.  
Therefore, $P \in \cL_\phi$, which establishes that $\overline{\gamma(\R)} \subseteq \cL_\phi$. 

We now establish that $\cL_\phi \subseteq \overline{\gamma(\R)}$. 
For each $p \in \simp \setminus \overline{\gamma(\R)}$,
there exists $\epsilon >0$ such that $p$ is not contained in the closed $33\epsilon$-neighborhood of $\gamma(\mathbb{R}),$ denoted by $N_{33\epsilon}(\gamma(\mathbb{R}))$.  
Now choose $K \geq 2$ so that $\epsilon_j < \epsilon$ for all $j \geq K$.
It follows from (\ref{eq:phicompbound}) that 
\[\phi[T^0_{N_0(K-1)-1}, \infty) =\bigcup_{j=K}^\infty \phi [T^0_{N_0(j-1)-1}, T^0_{N_0(j)-1}) \subseteq \bigcup_{j=K}^\infty N_{33\epsilon_j}(\gamma(t_j))  \subseteq N_{33\epsilon}(\gamma(\mathbb{R})). \]
Therefore, $p \not \in \cL_\phi$, which establishes that $\cL_\phi \subseteq \overline{\gamma(\R)}$. 
\end{proof}


\bibliography{}
\bibliographystyle{plain}

\end{document}

%% file: glued_up.pdf_tex
\begingroup%
  \makeatletter%
  \providecommand\color[2][]{%
    \errmessage{(Inkscape) Color is used for the text in Inkscape, but the package 'color.sty' is not loaded}%
    \renewcommand\color[2][]{}%
  }%
  \providecommand\transparent[1]{%
    \errmessage{(Inkscape) Transparency is used (non-zero) for the text in Inkscape, but the package 'transparent.sty' is not loaded}%
    \renewcommand\transparent[1]{}%
  }%
  \providecommand\rotatebox[2]{#2}%
  \ifx\svgwidth\undefined%
    \setlength{\unitlength}{304.76214241bp}%
    \ifx\svgscale\undefined%
      \relax%
    \else%
      \setlength{\unitlength}{\unitlength * \real{\svgscale}}%
    \fi%
  \else%
    \setlength{\unitlength}{\svgwidth}%
  \fi%
  \global\let\svgwidth\undefined%
  \global\let\svgscale\undefined%
  \makeatother%
  \begin{picture}(1,0.82693326)%
    \put(0,0){\includegraphics[width=\unitlength,page=1]{glued_up.pdf}}%
    \put(0.62651938,0.40029325){\color[rgb]{0,0,0}\makebox(0,0)[lb]{\smash{$\beta_i$}}}%
    \put(0.31433215,0.7884179){\color[rgb]{0,0,0}\makebox(0,0)[lt]{\begin{minipage}{0.32249972\unitlength}\raggedright $Y_i$\end{minipage}}}%
  \end{picture}%
\endgroup%